\documentclass[a4paper,twocolumn]{article}
\usepackage[fontsize=10pt]{fontsize}
\usepackage[T1]{fontenc}
\usepackage[utf8]{inputenc}
\usepackage{lmodern}
\usepackage{microtype}
\usepackage[a4paper,left=0.65in,right=0.65in,top=0.9in,bottom=0.95in,
            columnsep=0.25in]{geometry}
\usepackage{amsmath,amssymb,amsthm,mathtools}
\usepackage{booktabs,array}
\usepackage{caption}
\usepackage{enumitem}
\setlist{itemsep=0.2em,topsep=0.3em,parsep=0pt}
\setlist[enumerate,1]{label=\textup{(\arabic*)},leftmargin=*}
\usepackage[dvipsnames]{xcolor}
\definecolor{LinkNavy}{RGB}{20,45,95}
\usepackage[colorlinks=true,linkcolor=LinkNavy,citecolor=LinkNavy,
            urlcolor=LinkNavy,breaklinks=true]{hyperref}

\theoremstyle{plain}
\newtheorem{theorem}{Theorem}[section]
\newtheorem{lemma}[theorem]{Lemma}
\newtheorem{proposition}[theorem]{Proposition}
\newtheorem{corollary}[theorem]{Corollary}
\newtheorem{conjecture}[theorem]{Conjecture}
\theoremstyle{definition}

\newtheorem{question}[theorem]{Question}
\theoremstyle{remark}
\newtheorem{remark}[theorem]{Remark}

\newtheorem{mainthm}{Theorem}

\DeclareMathOperator{\Rom}{Rom}
\DeclareMathOperator{\cov}{cov}

\newcommand{\Z}{\mathbb{Z}}

\allowdisplaybreaks

\title{\bfseries A congruence obstruction to Roman's bound\\ for Zarankiewicz numbers}
\author{Ankan Sadhu\\[0.2em]
\small Department of Computer Science and Engineering,
\small Government College of Engineering and Ceramic Technology,\\
\small Kolkata, India\\
\small\texttt{ankan-sadhu-br24-3003@gcect.ac.in}}
\date{\small November 2026}

\begin{document}

\makeatletter
\twocolumn[%
\begin{@twocolumnfalse}
\vspace*{0.15em}
\begin{center}
  {\LARGE\bfseries A congruence obstruction to Roman's bound\\[0.25em]
   for Zarankiewicz numbers\par}
  \vskip 1.1em
  {\large Ankan Sadhu\par}
  \vskip 0.35em
{\small Department of Computer Science and Engineering\par}
  {\small Government College of Engineering and Ceramic Technology\par}
  {\small Kolkata, India\par}
  {\small\texttt{ankan-sadhu-br24-3003@gcect.ac.in}\par}
  \vskip 0.55em
  {\small November 2026\par}
  \vskip 0.30em
  {\footnotesize Author manuscript; published in \emph{Discrete Mathematics},
  \textbf{350}(4) (2027), 115552.\\
  \href{https://doi.org/10.1016/j.disc.2026.115552}{doi:10.1016/j.disc.2026.115552}\par}
  \vskip 0.24em
  {\footnotesize \textcopyright\ 2026 Elsevier B.V. This manuscript version is made available under\\
  \href{https://creativecommons.org/licenses/by-nc-nd/4.0/}{the CC BY-NC-ND 4.0 license}.\par}
\end{center}
\vskip 0.9em

\begin{abstract}\noindent
Let $z(m,n;s,t)$ be the largest number of ones in an $m \times n$ zero-one matrix with no $s \times t$ all-ones submatrix. Roman's 1975 inequality gives a general upper bound for $s \ge 3$, with improvements known in special parameter ranges. We show it is not attained on much of the range just below the design threshold $T = (t-1)\binom{m}{s}/(s+1)$.

For $n = T-c$ with $1 \le c \le sT/(s+2)$, Roman's bound equals the counting bound. Viewing columns as sets of row indices (points), attainment forces all but at most one column to have cardinality $s+1$ or $s+2$. These two sizes contribute multiples of $d = \gcd\bigl(s,\binom{s+1}{2}\bigr)$ to the coverage count at each point. At each point, the unused coverage is congruent to $\mu$ modulo $d$, with a correction for the possible exceptional column. With $r = c \bmod s$ and $\sigma(r)$ a slack depending only on $s$, the bound drops by at least one whenever $1 \le \sigma(r) < d$ and $m\mu \ne s\sigma(r)$, or $\mu \ne 0$ and $m\mu > s\sigma(r)$. The first needs $s$ odd and is confined to one residue class; the second covers the whole interval when $\mu \ne 0$ and $m$ is sufficiently large.

For $s = 4$, $t = 2$, $m = 28$ all $2730$ values of $n$ are covered. For odd $s$, when $c = (s+1)/2$ and an $s$-$(m,s+1,t-1)$ design exists we obtain the exact value $z = (s+1)(T-c)$. We also locate the obstruction relative to linear programming: the refined program of Davies, Gill and Horsley (2026) still attains the Roman vertex on an explicit family of such parameters.
\end{abstract}



\vskip 1.6em
\end{@twocolumnfalse}
]
\makeatother

\section{Introduction}

\subsection{The problem}

The general Zarankiewicz problem asks for the maximum number of ones in an $m \times n$ zero-one matrix with no $s \times t$ all-ones submatrix \cite{Zarankiewicz1951}. We write $z(m,n;s,t)$ for this maximum, and abbreviate $z(m,n;s,s)$ to $z_s(m,n)$. Seventy-five years on, the exact values are known only in scattered ranges. For $s = t = 2$ the K\H{o}v\'ari--S\'os--Tur\'an and Reiman theory \cite{KovariSosTuran1954,Reiman1958,HyltenCavallius1958} gives many sharp bounds, but for $s \ge 3$ the general picture is dominated by a single inequality due to Roman, and has been since 1975.

It helps to forget matrices at once. A matrix as above is the same thing as a list $C_1,\dots,C_n$ of subsets of $[m] = \{1,\dots,m\}$ --- the supports of its columns --- subject to the requirement that no $s$-subset of $[m]$ lie in $t$ or more of the $C_j$. Throughout we set
\[
  \lambda = t-1 ,
\]
so the requirement reads: every $s$-set is covered at most $\lambda$ times. The number of ones is $\sum_j |C_j|$, and for an $s$-set $S$ we write $\cov(S) = |\{ j : S \subseteq C_j \}|$.

Counting incidences between columns and $s$-sets gives the \emph{budget inequality}
\begin{equation}\label{eq:budget}
  \sum_{j=1}^{n} \binom{|C_j|}{s} \;=\; \sum_{|S| = s} \cov(S) \;\le\; \lambda \binom{m}{s} .
\end{equation}
A column of size $s+1$ contains exactly $s+1$ distinct $s$-sets. If all columns have this size and every $s$-set is covered exactly $\lambda$ times, then the number of columns is
\begin{equation}\label{eq:T}
  T \;=\; \frac{\lambda\binom{m}{s}}{s+1} ,
\end{equation}
the number of blocks of an $s$-$(m,s+1,\lambda)$ design. When such a design exists its blocks form a matrix with $n = T$ columns and $(s+1)T$ ones. We call $m$ \emph{admissible} for the given $s,t$ when $T \in \Z$, and we call $T$ the \emph{design threshold}.

\subsection{Roman's bound}

Roman \cite{Roman1975} proved that for every integer $k \ge s-1$,
\begin{equation}\label{eq:roman}
  z(m,n;s,t) \;\le\; \frac{\lambda\binom{m}{s}}{\binom{k}{s-1}} + \frac{(k+1)(s-1)}{s}\, n ,
\end{equation}
and we set $\Rom(m,n) = \min_{k \ge s-1} \lfloor \Rom_k(m,n) \rfloor$, where $\Rom_k(m,n)$ is the right-hand side of \eqref{eq:roman}. When the arguments are fixed, we abbreviate $\Rom_k(m,n)$ to $\Rom_k$. We use the normalisation of Davies, Gill and Horsley \cite{DaviesGillHorsley2026}; Roman writes $p$ in place of $k$. Inequality \eqref{eq:roman} remains a general upper bound for $s \ge 3$, although improvements are known in special parameter ranges.

As a function of $n$, the minimum in \eqref{eq:roman} is piecewise linear, with breakpoints at
\begin{equation}\label{eq:romanpoint}
  n_\ell \;=\; \frac{\lambda\binom{m}{s}}{\binom{\ell}{s}} , \qquad \ell \ge s ,
\end{equation}
where the value is $\ell\, n_\ell$. Following \cite[\S2]{DaviesGillHorsley2026} we call these the \emph{Roman points}. Note that $n_{s+1} = T$ and $n_{s+2} = 2T/(s+2)$. At an integral Roman point $n_\ell$ it is classical, and recorded explicitly in \cite[\S2]{DaviesGillHorsley2026}, that \eqref{eq:roman} is attained if and only if an $s$-$(m,\ell,\lambda)$ design exists.

Improvements on \eqref{eq:roman} are known in special ranges. {\v{C}}ul{\'i}k \cite{Culik1956} determined $z(m,n;s,t)$ exactly once $n$ is large relative to $m$; Guy \cite{Guy1969} and Irving \cite{Irving1978} gave refinements for small parameters; asymptotic results are due to M{\"o}rs \cite{Mors1981}, F{\"u}redi \cite{Furedi1996}, Nikiforov \cite{Nikiforov2010} and Conlon \cite{Conlon2022}. Links with bipartite Ramsey numbers and finite geometry are developed in \cite{GoddardHenningOellermann2000,DamasdiHegerSzonyi2013,CollinsRiasanovskyWallaceRadziszowski2016}, and the case $s = 2$ is settled in many instances by Chen, Horsley and Mammoliti \cite{ChenHorsleyMammoliti2024,ChenHorsleyMammoliti2024b}. Davies, Gill and Horsley \cite{DaviesGillHorsley2026} improve on \eqref{eq:roman} by linear programming, both in closed form (their Theorem~1.3, valid for $\max\{2,s^2-2s\} \le k \le m$) and numerically for small parameters with $s,t \ge 3$; their own summary is that Roman's bounds ``are still the best known upper bounds \dots\ in many cases''. For $s \ge 3$ we know of no general result that covers the whole interval just below the design threshold, and on the bulk of that interval \eqref{eq:roman} is still all we have; \S\ref{sec:lp} makes this precise by measuring how often the program of \cite{DaviesGillHorsley2026} improves on \eqref{eq:roman} there.

\subsection{What we prove}

Fix $s \ge 2$ and $t \ge 2$, put $\lambda = t-1$, let $m$ be admissible, and let $T$ be as in \eqref{eq:T}.

The modulus in our obstruction is the greatest common divisor of two point-incidence counts. A point of a column of size $s+1$ lies in $\binom{s}{s-1} = s$ of its $s$-subsets, and a point of a column of size $s+2$ lies in $\binom{s+1}{s-1} = \tfrac{s(s+1)}{2}$ of them. The largest modulus that cannot tell these two apart is
\begin{equation}\label{eq:d}
  d \;=\; \gcd\!\left(s,\ \binom{s+1}{2}\right) \;=\;
  \begin{cases}
    s, & \text{if $s$ is odd},\\[2pt]
    s/2, & \text{if $s$ is even},
  \end{cases}
\end{equation}
the evaluation because $\gcd\bigl(s, \tfrac{s}{2}(s+1)\bigr) = \tfrac{s}{2}\gcd(2,s+1) = \tfrac{s}{2}$ for even $s$. Set
\begin{equation}\label{eq:mu}
  \mu \;\equiv\; \lambda\binom{m-1}{s-1} \pmod{d}, \qquad 0 \le \mu < d .
\end{equation}
The modulus $d$ is trivial exactly when $s = 2$; Remark~\ref{rem:chm} explains what replaces it there.

The first result costs nothing but organises everything after it.

\begin{mainthm}\label{thm:A}
Let $s \ge 2$. For every integer $c$ with $1 \le c \le \dfrac{s}{s+2}\,T$,
\[
  \Rom(m,T-c) \;=\; (s+1)(T-c) + \left\lfloor \frac{2c}{s} \right\rfloor ,
\]
and the minimum of the right-hand sides of \eqref{eq:roman}, over integers $k \ge s-1$, is attained at $k = s+1$.
\end{mainthm}

Below the design threshold, then, Roman's bound is the naive counting bound, no more and no less. It carries no information beyond \eqref{eq:budget} and the convexity of $v \mapsto \binom{v}{s}$, so any improvement in this range has to come from a non-linear obstruction. We produce one.

\begin{mainthm}\label{thm:B}
Let $s \ge 3$, let $t \ge 2$, let $m \ge s+2$ be admissible, and let $d$ and $\mu$ be as in \eqref{eq:d} and \eqref{eq:mu}. Let $c$ be an integer with $1 \le c \le \tfrac{s}{s+2}T$, let $r$ be the residue of $c$ modulo $s$, and let $\sigma(r)$ be the slack of \eqref{eq:slack}. Suppose at least one of
\begin{align}
  &1 \le \sigma(r) < d \qquad\text{and}\qquad m\mu \ne s\,\sigma(r) , \label{eq:hyp1} \\
  &\mu \ne 0 \qquad\text{and}\qquad m\mu > s\,\sigma(r) \label{eq:hyp2}
\end{align}
holds. Then
\[
  z(m,T-c;s,t) \;\le\; \Rom(m,T-c) - 1 .
\]
\end{mainthm}

The two hypotheses are different obstructions, and it is worth saying what each delivers.

Hypothesis \eqref{eq:hyp1} asks that the leftover coverage be too small to distribute at all. By Theorem~\ref{thm:D}(1) this forces $s$ odd and $r = (s+1)/2$, and then $\sigma(r) = (s+1)/2$; the side condition $m\mu \ne s\sigma(r)$ fails only if $\mu \ne 0$ and $m = s(s+1)/(2\mu)$, so in particular it holds for every $m > s(s+1)/2$, and by Lemma~\ref{lem:design} it holds automatically whenever the relevant design exists. For fixed $m$ satisfying the side condition, this alternative covers one residue class modulo $s$ in an interval of length $sT/(s+2)$, and therefore
\[
  \frac{T}{s+2}+O(1) \;=\; \frac{\lambda\binom{m}{s}}{(s+1)(s+2)}+O(1)
\]
values of $n$, a count of order $\Theta(m^s)$ for fixed $s,t$.

Under hypothesis \eqref{eq:hyp2}, the nonzero point residue forces more leftover coverage than is available, even after accounting for a possible exceptional column. It says nothing about $r$. Since $\sigma(r) \le \sigma_{\max}$ is bounded in terms of $s$ alone by Theorem~\ref{thm:D}(2), while the left-hand side grows linearly in $m$, it holds for \emph{every} residue $r$ as soon as
\[
  m \;>\; \frac{s\,\sigma_{\max}}{\mu} ,
\]
a threshold depending only on $s$ and $\mu$. For such $m$ the conclusion covers every integer $n$ with $\tfrac{2}{s+2}T \le n < T$: not a progression inside an interval, but the whole interval, of length $\tfrac{s}{s+2}T = \Theta(m^s)$. Unlike \eqref{eq:hyp1}, this alternative is available for even $s$; see Theorem~\ref{thm:D}(3). Both hypotheses are sufficient rather than necessary, and Remark~\ref{rem:sharp} records what the method gives in full.

By Theorem~\ref{thm:A}, Roman's bound is the counting bound at every $n$ in the range, so the failure is not an artefact of a bad choice of $k$. The interaction with Roman points is recorded in Remark~\ref{rem:endpoint}.

If a design exists, the bound is exact at the first deficiency covered by \eqref{eq:hyp1}.

\begin{mainthm}\label{thm:C}
Let $s \ge 3$ be odd, let $t \ge 2$, let $m \ge s+2$ be admissible, take $c = (s+1)/2$, and suppose an $s$-$(m,s+1,\lambda)$ design exists. Then
\[
  \begin{aligned}
    z\!\left(m, T - \tfrac{s+1}{2}; s,t\right)
      &= (s+1)\!\left(T - \tfrac{s+1}{2}\right) \\
      &= \Rom\!\left(m, T - \tfrac{s+1}{2}\right) - 1 .
  \end{aligned}
\]
\end{mainthm}

By the existence theorems of Keevash \cite{Keevash2014} and of Glock, K\"uhn, Lo and Osthus \cite{GlockKuhnLoOsthus2023}, the design hypothesis holds for all sufficiently large $m$ satisfying the necessary divisibility conditions, for every odd $s$ and every $t$. For $s = 3$, $\lambda = 2$ no such machinery is needed: Hanani \cite{Hanani1960,Hanani1963} showed that a $3$-$(m,4,2)$ design exists precisely when $m \not\equiv 0 \pmod 3$ and $m \not\equiv 3 \pmod 4$; see \cite{ColbournDinitz2007} for existence tables.

Both hypotheses of Theorem~\ref{thm:B} are sharp in a strong sense.

\begin{mainthm}\label{thm:D}
Let $s \ge 3$, let $d$ be as in \eqref{eq:d}, and for $0 \le r < s$ put $\sigma(r) = (s+1)r - \binom{s+1}{2}\lfloor 2r/s \rfloor$.
\begin{enumerate}
\item If $s$ is odd, the unique residue $r$ with $1 \le \sigma(r) < d$ is $r = (s+1)/2$, and then $\sigma(r) = (s+1)/2$. If $s$ is even, no residue satisfies $1 \le \sigma(r) < d$. Hypothesis \eqref{eq:hyp1} is therefore an odd-$s$ phenomenon, confined to a single residue class.
\item The slack is bounded in terms of $s$ alone:
\[
  \sigma_{\max} \;:=\; \max_{0 \le r < s} \sigma(r) \;=\;
  \begin{cases} \dfrac{s^2-1}{2}, & \text{if }s\text{ is odd},\\[6pt] \dfrac{(s+1)(s-2)}{2}, & \text{if }s\text{ is even}.\end{cases}
\]
\item Hypothesis \eqref{eq:hyp2} is available for both parities of $s$. It requires $\mu \ne 0$ and hence $d \nmid \lambda$; conversely, if $d \nmid \lambda$ then $\mu \equiv \lambda \not\equiv 0 \pmod{d}$ for every $m \equiv s \pmod{(s-1)!\,d}$ that is admissible.
\item The modulus $d$ is trivial exactly when $s = 2$.
\end{enumerate}
\end{mainthm}

Part (1) shows that the small-slack condition \eqref{eq:hyp1} is restricted to odd $s$, while part (3) shows that the nonzero-residue condition \eqref{eq:hyp2} can also apply for even $s$. Remark~\ref{rem:even} records the smallest even instances.

Specialising to $s = t = 3$: here $d = 3$ and $\mu \equiv (m-1)(m-2) \pmod 3$, so $\mu = 2$ when $3 \mid m$ and $\mu = 0$ otherwise. The two hypotheses split along divisibility of $m$ by $3$, and between them cover every admissible $m$.

\begin{corollary}\label{cor:E}
Let $m \ge 5$ satisfy $m \not\equiv 3 \pmod 4$ and put $T = m(m-1)(m-2)/12$.
\begin{enumerate}
\item For every $c \equiv 2 \pmod 3$ with $2 \le c \le 3T/5$,
\[
  \begin{aligned}
    z(m,T-c;3,3) &\;\le\; 4(T-c) + \tfrac{2c-1}{3} - 1 \\
      &\;=\; \Rom(m,T-c) - 1 .
  \end{aligned}
\]
If moreover $3 \nmid m$, then for $c = 2$ this is an equality, and
\[
  z(m,T-2;3,3) \;=\; 4(T-2) \;=\; \Rom(m,T-2) - 1 .
\]
\item If $3 \mid m$ and $m \ge 9$, then for \emph{every} integer $c$ with $1 \le c \le 3T/5$,
\[
  z(m,T-c;3,3) \;\le\; \Rom(m,T-c) - 1 .
\]
\end{enumerate}
\end{corollary}

Part (2) follows from \eqref{eq:hyp2}: here $\mu = 2$ and $\sigma(r) \le 4$, so $m\mu > s\sigma(r)$ reduces to $2m > 12$. The bound $m \ge 9$ is sharp: at $m = 6$ one has $z(6,9;3,3) = 36 = \Rom(6,9)$ and $z(6,6;3,3) = 26 = \Rom(6,6)$, so Roman's bound really is attained there \cite{Tan2022}.

Table~\ref{tab:values} records the first few instances of part (1). For $m = 10$, Roman's bound fails at all twelve values $n = 58,55,52,\dots,25$. Part (2) is stronger where it applies: for $m = 9$ it gives all twenty-five values $n = 17,18,\dots,41$, consecutively.

\begin{table*}[t]
\centering
\renewcommand{\arraystretch}{1.15}
\begin{tabular}{r r r r r r}
\toprule
$m$ & $T$ & \#$\{c\}$ & $n = T-2$ & $\Rom(m,T-2)$ & $z(m,T-2;3,3)$ \\
\midrule
$5$  & $5$    & $1$   & $3$    & $13$   & $12$   \\
$8$  & $28$   & $5$   & $26$   & $105$  & $104$  \\
$10$ & $60$   & $12$  & $58$   & $233$  & $232$  \\
$13$ & $143$  & $28$  & $141$  & $565$  & $564$  \\
$14$ & $182$  & $36$  & $180$  & $721$  & $720$  \\
$16$ & $280$  & $56$  & $278$  & $1113$ & $1112$ \\
$17$ & $340$  & $68$  & $338$  & $1353$ & $1352$ \\
$20$ & $570$  & $114$ & $568$  & $2273$ & $2272$ \\
$22$ & $770$  & $154$ & $768$  & $3073$ & $3072$ \\
$25$ & $1150$ & $230$ & $1148$ & $4593$ & $4592$ \\
$26$ & $1300$ & $260$ & $1298$ & $5193$ & $5192$ \\
\bottomrule
\end{tabular}
\medskip
\caption{The case $s = t = 3$, $3 \nmid m$. The column \#$\{c\}$ counts the $c \equiv 2 \pmod 3$ with $2 \le c \le 3T/5$, that is, the number of columns $n$ at which Corollary~\ref{cor:E} shows Roman's bound fails. The last two columns give the smallest such instance, $c = 2$, where Theorem~\ref{thm:C} supplies the exact value.}
\label{tab:values}
\end{table*}

\subsection{Relation to previous work}\label{ssec:prev}

Non-attainment of \eqref{eq:roman} is not new, and infinite families of non-attainment are not new either. At an integral Roman point the design criterion quoted in \S1.2 settles the matter. For $s = t = 3$ the relevant Roman point is $n = T$ and the relevant design is a $3$-$(m,4,2)$ design, whose existence requires $3 \nmid m$; so \eqref{eq:roman} already fails at $(m,T)$ for every admissible $m$ divisible by $3$, an infinite family obtainable in two lines. Individual values below Roman's bound are also known: Tan's satisfiability computations \cite{Tan2022} give, among many others, $z_3(6,8) = 32$, $z_3(6,10) = 39$ and $z_3(3,5) = 12$; Davies, Gill and Horsley \cite{DaviesGillHorsley2026} improve on \eqref{eq:roman} for a number of small parameter sets with $s = t = 3$ by linear programming; and the unrefereed preprint \cite{BhanNobiliLanger2026} reports $z_3(11,21) = 116$, $z_3(11,22) = 121$ and $z_3(12,22) = 132$.

The tools are standard. The quantity $\rho(S) = \lambda - \cov(S)$ of \S\ref{sec:rigidity} is the multiplicity of $S$ in the leave of the partial design formed by the columns, and divisibility constraints on leaves are a routine device in design theory \cite{ColbournDinitz2007}. The point-link double count of Lemma~\ref{lem:link} goes back to Guy \cite{Guy1969} and is used by Tan \cite[Thm.~2.3]{Tan2022} to bound packing numbers. What is new is not the congruence but where it is applied.

The closest antecedent is a second paper of Chen, Horsley and Mammoliti \cite{ChenHorsleyMammoliti2024b}, which determines $Z_{2,2}(m,n)$ for every $n$ above the triple system threshold $\tfrac13\binom{m}{2}$ and, for large $m$, for every $n$ down to $\tfrac16\binom{m}{2} + O(m)$. Their Lemmas 2.5 and 2.6 have the same three-step shape as our Theorem~\ref{thm:B}: pin the degree profile down with a counting identity, then contradict the surviving profile with a divisibility obstruction on a derived graph. The strategy is theirs, not ours. Two things differ. Their case is $s = 2$, where the obstruction is a parity statement --- the defect of a linear $3$-hypergraph on an odd number of vertices is an even graph --- and parity bites only finitely often, so their exceptional set consists of the deficiencies in $\{1,2,3,4\}$ or $\{0,1,3\}$, a set bounded independently of $m$. Ours is a congruence modulo $d$, available for every $s \ge 3$ of either parity, and it bites at every deficiency in a fixed residue class --- or, when $\mu \ne 0$ and $m > s\sigma_{\max}/\mu$, at every deficiency in the whole range --- so the exceptional set has size $\Theta(m^s)$. Their two bounds $U^{+}$ and $U^{-}$ are the cases $k = s$ and $k = s+1$ of $\Rom_k$, and the fact that $U^{-}$ binds below the threshold is the $s = 2$ instance of Theorem~\ref{thm:A}.

Three things are new. First, Theorem~\ref{thm:A} identifies Roman's bound below the threshold with the elementary counting bound. For $s = 2$ this is implicit in \cite{ChenHorsleyMammoliti2024b}; we add the statement for all $s$ and $t$, with the explicit range $c \le \tfrac{s}{s+2}T$ on which it holds, and it is this identity that makes the remaining proofs short. Second, the classical design argument operates only \emph{at} Roman points and only when a design fails to exist, whereas Theorem~\ref{thm:B} operates at a positive-density set of $n$ strictly below the threshold; via hypothesis \eqref{eq:hyp1} it does so precisely when the design \emph{does} exist, which is the case where the classical argument is silent. Third, \S\ref{sec:lp} locates the obstruction relative to linear programming. Proposition~\ref{prop:subsetlp} shows that the relaxation over all $2^{[m]}$ subset variables collapses, after $S_m$-symmetrisation, to the counting bound; with Theorem~\ref{thm:A} this says that Roman's bound, the counting bound and the full subset relaxation agree throughout the range, so an improvement here must be an integrality obstruction. The refined program of \cite{DaviesGillHorsley2026} adds genuinely new inequalities and does beat \eqref{eq:roman} at some parameters in the range; Proposition~\ref{prop:lp} determines exactly when its optimum is still the Roman value, and the sub-family it describes contains the parameters of Theorem~\ref{thm:C}.

\subsection{Organisation}

In Section~\ref{sec:counting} we prove Theorem~\ref{thm:A}. In Section~\ref{sec:rigidity} we prove Theorem~\ref{thm:D} and classify the degree profiles that can meet the bound. The link congruence and the design lemma are established in Section~\ref{sec:congruence}; these are then used in Section~\ref{sec:main} to prove Theorem~\ref{thm:B}. Section~\ref{sec:exact} gives the construction proving Theorem~\ref{thm:C}, and Section~\ref{sec:sharp} discusses sharpness, including even $s$ and an example showing that the failure set need not be an interval. We compare the obstruction with linear programming in Section~\ref{sec:lp} and collect open problems in Section~\ref{sec:open}.

\section{The counting bound}\label{sec:counting}

Throughout this section $s \ge 2$, $\lambda = t-1$, $m$ is admissible, and $C_1,\dots,C_n$ is a list of subsets of $[m]$ in which every $s$-set is covered at most $\lambda$ times. Put $d_j = |C_j|$.

It is natural to measure degrees from $s+1$. Set
\[
  e_j = d_j - (s+1), \qquad g(e) = \binom{s+1+e}{s} - (s+1) ,
\]
so $g(0) = 0$ and \eqref{eq:budget} becomes $\sum_j g(e_j) \le \lambda\binom{m}{s} - (s+1)n$.

\begin{lemma}[Two-slope convexity]\label{lem:convex}
For every integer $e \ge -(s+1)$,
\[
  \begin{gathered}
    g(e) \;\ge\; \binom{s+1}{2} e^{+} \;-\; s\,e^{-} , \\
    e^{+} = \max(e,0), \quad e^{-} = \max(-e,0) ,
  \end{gathered}
\]
with equality if and only if $e \in \{-1,0,1\}$.
\end{lemma}

\begin{proof}
The function $v \mapsto \binom{v}{s}$ is convex, so it suffices to check the two steps adjacent to $v = s+1$ and confirm all later steps are steeper.

Upwards, $g(1) = \binom{s+2}{2} - (s+1) = \binom{s+1}{2}$, the asserted slope. The later increments are $g(e+1) - g(e) = \binom{s+1+e}{s-1}$, and for $e \ge 1$ this is at least $\binom{s+2}{3} = \tfrac{(s+2)(s+1)s}{6} \ge \tfrac{3(s+1)s}{6} = \binom{s+1}{2}$, strictly for $s \ge 2$. Hence $g(e) \ge \binom{s+1}{2}e$ for $e \ge 1$, strictly for $e \ge 2$.

Downwards, $g(-1) = \binom{s}{s} - (s+1) = -s$, again the asserted value. For $e \le -2$ we have $\binom{s+1+e}{s} = 0$, so $g(e) = -(s+1)$, whereas the asserted lower bound is $se \le -2s$; and $-(s+1) > -2s$ for $s \ge 2$.
\end{proof}

\begin{lemma}[Basic inequality]\label{lem:basic}
Let $n = T-c$ with $0 \le c < T$, and suppose the matrix with column supports $C_1,\dots,C_n$ has exactly $(s+1)n + a$ ones. Put $L = \sum_j e_j^{-}$. Then
\begin{equation}\label{eq:star}
  \binom{s+1}{2}\,a \;+\; \binom{s}{2}\,L \;\le\; (s+1)\,c .
\end{equation}
In particular $a \le 2c/s$, and therefore
\begin{equation}\label{eq:counting}
  z(m,T-c;s,t) \;\le\; (s+1)(T-c) + \left\lfloor \frac{2c}{s} \right\rfloor .
\end{equation}
\end{lemma}

\begin{proof}
The number of ones is $\sum_j d_j = (s+1)n + \sum_j e_j$, so $a = \sum_j e_j$. With $H = \sum_j e_j^{+}$ we have $H - L = a$. Combining \eqref{eq:budget} with Lemma~\ref{lem:convex},
\[
  \begin{aligned}
    (s+1)n + \binom{s+1}{2}H - sL
      &\;\le\; \sum_j \binom{d_j}{s} \\
      &\;\le\; \lambda\binom{m}{s} = (s+1)T \\
      &= (s+1)n + (s+1)c .
  \end{aligned}
\]
Substituting $H = a + L$ and using $\binom{s+1}{2} - s = \binom{s}{2}$ gives \eqref{eq:star}. Discarding the non-negative term $\binom{s}{2}L$ yields $\binom{s+1}{2}a \le (s+1)c$, that is $a \le 2c/s$.
\end{proof}

We now show \eqref{eq:counting} is Roman's bound. Write $\Rom_k(m,T-c) = A(k)T - B(k)c$, where substituting $\lambda\binom{m}{s} = (s+1)T$ into \eqref{eq:roman} gives
\begin{equation}\label{eq:AB}
  A(k) = \frac{s+1}{\binom{k}{s-1}} + \frac{(k+1)(s-1)}{s} , \qquad B(k) = \frac{(k+1)(s-1)}{s} .
\end{equation}

\begin{lemma}\label{lem:AB}
With $A$ as in \eqref{eq:AB} we have $A(s) = A(s+1) = s+1$ and $A(s-1) = 2s$.
\end{lemma}

\begin{proof}
Since $\binom{s}{s-1} = s$ and $\binom{s+1}{s-1} = s(s+1)/2$,
\[
  \begin{aligned}
    A(s) &= \frac{s+1}{s} + \frac{(s+1)(s-1)}{s} \\
      &= \frac{(s+1)\bigl(1 + (s-1)\bigr)}{s} = s+1 , \\
    A(s+1) &= \frac{2}{s} + \frac{(s+2)(s-1)}{s} = s+1 ,
  \end{aligned}
\]
and $A(s-1) = (s+1) + \tfrac{s(s-1)}{s} = 2s$.
\end{proof}

The coincidence $A(s) = A(s+1)$ is the analytic shadow of the design threshold: at $n = T$ the choices $k = s$ and $k = s+1$ give the same value $(s+1)T$, which is exactly the number of ones in a design. Which of them wins below the threshold is decided by the linear term, and $B(s+1) > B(s)$.

\begin{proof}[Proof of Theorem~\ref{thm:A}]
Take $k = s+1$. By Lemma~\ref{lem:AB} and \eqref{eq:AB},
\[
  \begin{aligned}
    \Rom_{s+1}(m,T-c) &= (s+1)T - \frac{(s+2)(s-1)}{s}c \\
      &= (s+1)(T-c) + \frac{2c}{s},
  \end{aligned}
\]
whose floor is $(s+1)(T-c) + \lfloor 2c/s \rfloor$ because $(s+1)(T-c) \in \Z$. This is exactly the counting bound \eqref{eq:counting} of Lemma~\ref{lem:basic}.

It remains to show $\Rom_k \ge \Rom_{s+1}$ for every other $k$ in the stated range of $c$. We have
\[
  \Rom_k - \Rom_{s+1} = \bigl(A(k)-s-1\bigr)T - \bigl(B(k)-B(s+1)\bigr)c .
\]
For $k = s$, $A(s) - s - 1 = 0$ and $B(s+1) - B(s) = (s-1)/s > 0$, so the difference is $\tfrac{s-1}{s}c > 0$. For $k = s-1$, $A(s-1)-s-1 = s-1 > 0$ and $B(s-1) - B(s+1) = -2(s-1)/s < 0$, so the difference is $(s-1)T + \tfrac{2(s-1)}{s}c > 0$.

Now let $k = s+1+i$ with $i \ge 1$, so $B(k) - B(s+1) = i(s-1)/s > 0$. Using $A(s+1) = s+1$ and $\binom{s+1}{2}^{-1} = 2/(s(s+1))$,
\begin{equation}\label{eq:ratio}
\begin{aligned}
  R(i) \;&:=\; \frac{A(k)-s-1}{B(k)-B(s+1)} \\
  &\;=\; 1 - \frac{2}{i(s-1)}
     + \frac{s(s+1)}{i(s-1)\binom{s+1+i}{s-1}} ,
\end{aligned}
\end{equation}
and $\Rom_k \ge \Rom_{s+1}$ holds precisely when $c \le R(i)\,T$. So it suffices to prove $R(i) \ge \tfrac{s}{s+2}$ for all $i \ge 1$.

At $i = 1$ we have $\binom{s+2}{s-1} = \tfrac{(s+2)(s+1)s}{6}$, so
\[
  \begin{aligned}
    R(1) &= 1 - \frac{2}{s-1} + \frac{6}{(s-1)(s+2)} \\
      &= 1 - \frac{2(s-1)}{(s-1)(s+2)} = \frac{s}{s+2} ,
  \end{aligned}
\]
with equality; this is what fixes the range. For $i \ge 2$ we treat three cases.

If $s = 2$, then $s - 1 = 1$ and $\binom{3+i}{1} = i+3$, so \eqref{eq:ratio} evaluates exactly:
\[
  \begin{aligned}
    R(i) &= 1 - \frac{2}{i} + \frac{6}{i(i+3)} \\
      &= \frac{i(i+3) - 2(i+3) + 6}{i(i+3)} \\
      &= \frac{i^2+i}{i(i+3)} = \frac{i+1}{i+3} ,
  \end{aligned}
\]
which is increasing in $i$ with $R(1) = \tfrac12 = \tfrac{s}{s+2}$. So $R(i) \ge \tfrac{s}{s+2}$ for all $i \ge 1$.

If $s \ge 4$, discard the positive third term in \eqref{eq:ratio}: for $i \ge 2$,
\[
  R(i) \;\ge\; 1 - \frac{2}{i(s-1)} \;\ge\; 1 - \frac{1}{s-1} \;=\; \frac{s-2}{s-1} \;\ge\; \frac{s}{s+2} ,
\]
the last step because $(s-2)(s+2) - s(s-1) = s-4 \ge 0$.

If $s = 3$, then $R(i) = 1 - \tfrac1i + \tfrac{6}{i\binom{4+i}{2}}$. At $i = 2$ this is $\tfrac12 + \tfrac{6}{30} = \tfrac{7}{10}$, and for $i \ge 3$ it exceeds $1 - \tfrac1i \ge \tfrac23$; both are at least $\tfrac35 = \tfrac{s}{s+2}$.

So $R(i) \ge \tfrac{s}{s+2}$ throughout, and $\Rom_k \ge \Rom_{s+1}$ whenever $c \le \tfrac{s}{s+2}T$.
\end{proof}

\section{Rigidity}\label{sec:rigidity}

Write $c = qs + r$ with $0 \le r < s$ and define the \emph{slack}
\begin{equation}\label{eq:slack}
  \sigma(r) \;=\; (s+1)r - \binom{s+1}{2}\left\lfloor \frac{2r}{s} \right\rfloor .
\end{equation}

\begin{proof}[Proof of Theorem~\ref{thm:D}]
(1) If $2r < s$ then $\lfloor 2r/s \rfloor = 0$ and $\sigma(r) = (s+1)r$, which is $0$ at $r = 0$ and at least $s+1 > s \ge d$ for $r \ge 1$; no residue with $1 \le \sigma(r) < d$ arises. Otherwise $s \le 2r < 2s$, so $\lfloor 2r/s \rfloor = 1$ and, putting $w = 2r-s \ge 0$,
\[
  \sigma(r) \;=\; (s+1)r - \frac{s(s+1)}{2} \;=\; \frac{(s+1)w}{2} .
\]
If $s$ is odd then $d = s$, and $1 \le \sigma(r) < s$ forces $w \ge 1$ and $w < 2s/(s+1) < 2$, hence $w = 1$; as $w = 2r-s$ has the parity of $s$ this is achievable, and gives $r = (s+1)/2$ with $\sigma(r) = (s+1)/2$. If $s$ is even then $d = s/2$ and $w$ is even, so either $w = 0$ and $\sigma(r) = 0$, or $w \ge 2$ and $\sigma(r) \ge s+1 > s/2 = d$; no residue qualifies.

(2) On the range $2r < s$ the quantity $\sigma(r) = (s+1)r$ is largest at the greatest admissible $r$, namely $r = \tfrac{s-1}{2}$ for odd $s$ and $r = \tfrac{s}{2}-1$ for even $s$, where $\sigma(r)$ has values $\tfrac{s^2-1}{2}$ and $\tfrac{(s+1)(s-2)}{2}$, respectively. On the range $2r \ge s$ we have $\sigma(r) = \tfrac{(s+1)w}{2}$ with $w = 2r-s \le s-2$; here $\sigma(r)$ is largest at $r = s-1$, where it equals $\tfrac{(s+1)(s-2)}{2}$. For odd $s$ the first range wins, since $\tfrac{s^2-1}{2} - \tfrac{(s+1)(s-2)}{2} = \tfrac{s+1}{2} > 0$; for even $s$ the two agree.

(3) That $\mu \ne 0$ requires $d \nmid \lambda$ is immediate, since $d \mid \lambda$ would give $d \mid \lambda\binom{m-1}{s-1}$. Conversely suppose $d \nmid \lambda$ and $m \equiv s \pmod{(s-1)!\,d}$. Then $m-1-i \equiv s-1-i \pmod{(s-1)!\,d}$ for each $0 \le i \le s-2$, so multiplying these $s-1$ congruences gives
\[
  \prod_{i=0}^{s-2}(m-1-i) \;\equiv\; (s-1)! \pmod{(s-1)!\,d} .
\]
Dividing by $(s-1)!$, which divides both sides and the modulus, yields $\binom{m-1}{s-1} \equiv 1 \pmod{d}$, whence $\mu \equiv \lambda \not\equiv 0 \pmod d$. This proves the asserted congruence for every admissible $m$ in the stated class.

(4) For $s \ge 3$ we have $d \ge s/2 \ge 3/2$, so $d \ge 2$; and $d = \gcd(2,3) = 1$ for $s = 2$.
\end{proof}

It is convenient to measure every column against columns of sizes $s+1$ and $s+2$ at once. For $0 \le v \le m$ put
\begin{equation}\label{eq:pen}
  p(v) \;=\; \binom{v}{s} - (s+1) - \binom{s+1}{2}\bigl(v - s - 1\bigr) ,
\end{equation}
the amount by which a column of size $v$ overspends the coverage budget relative to the affine interpolation through the sizes $s+1$ and $s+2$.

\begin{lemma}[The penalty function]\label{lem:pen}
The function $p$ of \eqref{eq:pen} is convex, vanishes at $v = s+1$ and $v = s+2$, is strictly positive at every other $v$, and is strictly decreasing for $v \le s+1$ and strictly increasing for $v \ge s+2$. Moreover
\[
  p(s) = \binom{s}{2}, \qquad p(s-1) = s^2-1, \qquad p(s+3) = \binom{s+1}{3} .
\]
\end{lemma}

\begin{proof}
The map $v \mapsto \binom{v}{s}$ has successive differences $\binom{v}{s-1}$, which are non-decreasing, so it is convex; $p$ subtracts an affine function and is convex too. Next,
\[
  \begin{gathered}
    p(s+1) = (s+1)-(s+1)-0 = 0, \\
    p(s+2) = \binom{s+2}{2} - (s+1) - \binom{s+1}{2} = 0 ,
  \end{gathered}
\]
the latter because $\binom{s+2}{2} - \binom{s+1}{2} = s+1$. The successive differences are
\[
  p(v+1)-p(v)=\binom{v}{s-1}-\binom{s+1}{2}.
\]
They are negative for $v \le s$, zero at $v=s+1$, and positive for $v \ge s+2$. Thus $p$ is strictly decreasing for $v \le s+1$ and strictly increasing for $v \ge s+2$. Together with the two zeros, this proves strict positivity at every other integer. For the three values,
\[
  \begin{aligned}
    p(s) &= 1 - (s+1) + \binom{s+1}{2} \\
      &= \binom{s+1}{2} - s = \binom{s}{2}, \\
    p(s-1) &= -(s+1) + 2\binom{s+1}{2} = s^2-1 ,
  \end{aligned}
\]
and, using $\binom{s+3}{s} = \binom{s+3}{3}$,
\begin{align*}
  p(s+3) &= \binom{s+3}{3} - (s+1) - 2\binom{s+1}{2} \\
    &= (s+1)\left(\frac{(s+3)(s+2)}{6} - 1 - s\right) \\
    &= \frac{(s+1)s(s-1)}{6} = \binom{s+1}{3} . \qedhere
\end{align*}
\end{proof}

The \emph{degree profile} of a list of columns records how many columns have each size.

\begin{lemma}[Classification of optimal profiles]\label{lem:rigid}
Let $n = T - c$ and suppose a valid system on $n$ columns has exactly $(s+1)n + a$ ones, where $a = \lfloor 2c/s \rfloor$. Write $d_1,\dots,d_n$ for the column sizes and $\rho(S) = \lambda - \cov(S) \ge 0$ for $|S| = s$. Then the total slack
\begin{equation}\label{eq:Dpen}
  D \;:=\; \sum_{|S|=s} \rho(S) \;=\; \sigma(r) - \sum_{j=1}^{n} p(d_j) \;\ge\; 0 ,
\end{equation}
so $\sum_j p(d_j) \le \sigma(r)$. Call a column \emph{exceptional} if its size lies outside $\{s+1,s+2\}$. Then:
\begin{enumerate}
\item if $s$ is even there are no exceptional columns: every column has size $s+1$ or $s+2$, exactly $a$ of them have size $s+2$, and $D = \sigma(r)$;
\item if $s$ is odd there is at most one exceptional column. If $s \ge 5$ its size is $s$; if $s = 3$ its size is $3$ or $6$. In either case $D = \sigma(r) - p(v)$, where $v$ is the size of the exceptional column if there is one, and $D = \sigma(r)$ if there is not.
\end{enumerate}
\end{lemma}

\begin{proof}
Summing \eqref{eq:pen} over all columns and using $\sum_j d_j = (s+1)n + a$ gives
$\sum_j p(d_j) = \sum_j \binom{d_j}{s} - (s+1)n - \binom{s+1}{2}a$.
On the other hand $D = \lambda\binom{m}{s} - \sum_j \binom{d_j}{s} = (s+1)T - \sum_j \binom{d_j}{s}$ by \eqref{eq:T}. Adding,
\[
  \begin{aligned}
    \sum_j p(d_j) + D
      &\;=\; (s+1)(T-n) - \binom{s+1}{2}a \\
      &\;=\; (s+1)c - \binom{s+1}{2}a .
  \end{aligned}
\]
Write $c = qs+r$, so $a = 2q + \lfloor 2r/s \rfloor$. The terms in $q$ cancel, since $(s+1)qs = \binom{s+1}{2}\cdot 2q$, and what remains is $(s+1)r - \binom{s+1}{2}\lfloor 2r/s \rfloor = \sigma(r)$. This is \eqref{eq:Dpen}, and $D \ge 0$ because every $\rho(S) \ge 0$.

Let $k$ be the number of exceptional columns. By Lemma~\ref{lem:pen} each exceptional column contributes at least $\min\{p(s),p(s+3)\} = \binom{s}{2}$, the minimum being $p(s)$ because $\binom{s+1}{3} \ge \binom{s}{2}$ for $s \ge 2$. Hence $k\binom{s}{2} \le \sum_j p(d_j) \le \sigma(r) \le \sigma_{\max}$, where $\sigma_{\max}$ is evaluated in Theorem~\ref{thm:D}(2).

If $s$ is even then $\sigma_{\max} = \tfrac{(s+1)(s-2)}{2} < \tfrac{s(s-1)}{2} = \binom{s}{2}$, because $(s+1)(s-2) = s^2-s-2 < s^2-s$. So $k = 0$; every column has size $s+1$ or $s+2$, and since $\sum_j (d_j - s-1) = a$ exactly $a$ of them have size $s+2$. This proves (1).

If $s$ is odd then $\sigma_{\max} = \tfrac{s^2-1}{2} < s(s-1) = 2\binom{s}{2}$, since $s^2-1 < 2s^2-2s$ rearranges to $(s-1)^2 > 0$. So $k \le 1$. Suppose $k = 1$, with exceptional column of size $v$; then $p(v) \le \sigma_{\max} = \tfrac{(s+1)(s-1)}{2}$. Since $p$ increases as $v$ moves away from $\{s+1,s+2\}$, it is enough to exclude $v = s-1$ on the left and, when $s \ge 5$, $v = s+3$ on the right. For the first, $p(s-1) = s^2-1 > \tfrac{s^2-1}{2}$. Here $s+3$ is the smallest size above $s+2$, so excluding it also excludes every larger size. For this comparison,
\[
  \frac{p(s+3)}{\sigma_{\max}} \;=\; \frac{(s+1)s(s-1)/6}{(s+1)(s-1)/2} \;=\; \frac{s}{3} ,
\]
which exceeds $1$ precisely when $s > 3$. So for $s \ge 5$ the only possible exceptional size is $v = s$, while for $s = 3$ both $v = 3$ and $v = 6$ remain, the latter with $p(6) = \binom{4}{3} = 4 = \sigma_{\max}$. The formula for $D$ is \eqref{eq:Dpen}.
\end{proof}

In the situation relevant to Theorem~\ref{thm:C}, namely $s$ odd and $c \equiv \tfrac{s+1}{2} \pmod s$,
\begin{equation}\label{eq:asigma}
  a = \left\lfloor \frac{2c}{s} \right\rfloor = 2q+1, \qquad \sigma(r) = \frac{s+1}{2} ,
\end{equation}
and the profile is rigid: $\tfrac{s+1}{2} < \binom{s}{2}$ for $s \ge 3$, so Lemma~\ref{lem:rigid} permits no exceptional column and $D = \tfrac{s+1}{2}$.

\section{The link congruence}\label{sec:congruence}

For $x \in [m]$ define the \emph{local slack}
\[
  \rho_x \;=\; \sum_{\substack{|S| = s \\ x \in S}} \rho(S) .
\]
Each $s$-set contains exactly $s$ points, so summing over $x$ counts each $\rho(S)$ exactly $s$ times:
\begin{equation}\label{eq:sum}
  \sum_{x \in [m]} \rho_x \;=\; s\,D .
\end{equation}
Also $\rho_x$ is a partial sum of $\sum_S \rho(S) = D$, whence
\begin{equation}\label{eq:cap}
  0 \;\le\; \rho_x \;\le\; D \qquad \text{for every } x \in [m] .
\end{equation}

\begin{lemma}[Link congruence]\label{lem:link}
Assume the situation of Lemma~\ref{lem:rigid} and let $d$ be as in \eqref{eq:d}. Let $C_0$ be the exceptional column if one exists, say of size $v$, and put $w = \binom{v-1}{s-1} \bmod d$; if there is no exceptional column set $C_0 = \emptyset$. Then for every $x \in [m]$,
\[
  \rho_x \;\equiv\;
  \begin{cases}
    \mu - w \pmod{d}, & x \in C_0, \\[2pt]
    \mu \pmod{d}, & x \notin C_0 .
  \end{cases}
\]
In particular, if there is no exceptional column then $\rho_x \equiv \mu \pmod d$ for every $x \in [m]$.
\end{lemma}

\begin{proof}
Count in two ways the pairs $(S,j)$ with $|S| = s$, $x \in S$ and $S \subseteq C_j$.

Grouping by $S$, and using that there are $\binom{m-1}{s-1}$ such $S$, the count is $\sum_{S \ni x} \cov(S) = \lambda\binom{m-1}{s-1} - \rho_x$.

Grouping by $j$, a column $C_j$ containing $x$ contributes the number of $s$-subsets of $C_j$ through $x$, namely $\binom{d_j - 1}{s-1}$. By Lemma~\ref{lem:rigid} every column except possibly one has size $s+1$ or $s+2$, contributing
\[
  \binom{s}{s-1} = s \qquad\text{or}\qquad \binom{s+1}{s-1} = \binom{s+1}{2} = \frac{s(s+1)}{2}
\]
respectively, and both are divisible by $d$ by \eqref{eq:d}. The exceptional column, if it exists and contains $x$, contributes $\binom{v-1}{s-1} \equiv w \pmod d$. Hence the total is congruent modulo $d$ to $w$ if $x \in C_0$ and to $0$ otherwise, and comparing the two expressions gives the claim.
\end{proof}

The next lemma is short but it is what keeps Theorem~\ref{thm:B} consistent with the classical design criterion, and it is used twice below.

\begin{lemma}[Design lemma]\label{lem:design}
Let $s \ge 3$ and let $k \in \{s+1,s+2\}$. If an $s$-$(m,k,\lambda)$ design exists, then $\mu = 0$.
\end{lemma}

\begin{proof}
In an $s$-$(m,k,\lambda)$ design the number of blocks through a fixed point is
\[
  \lambda_1 \;=\; \frac{\lambda\binom{m-1}{s-1}}{\binom{k-1}{s-1}} \;\in\; \Z ,
\]
so $\binom{k-1}{s-1}$ divides $\lambda\binom{m-1}{s-1}$. For $k = s+1$ we have $\binom{k-1}{s-1} = \binom{s}{s-1} = s$, and for $k = s+2$ we have $\binom{k-1}{s-1} = \binom{s+1}{s-1} = \binom{s+1}{2}$. By \eqref{eq:d}, $d$ divides both of these. Hence $d \mid \lambda\binom{m-1}{s-1}$, that is $\mu = 0$.
\end{proof}

\begin{remark}\label{rem:designhyp}
Lemma~\ref{lem:design} has two consequences worth stating. First, hypothesis \eqref{eq:hyp2} is never available when an $s$-$(m,s+1,\lambda)$ or $s$-$(m,s+2,\lambda)$ design exists, since it requires $\mu \ne 0$. Second, the side condition in \eqref{eq:hyp1} is automatic in that situation: $m\mu = 0$ while $s\sigma(r) \ge s \ge 3$. So it is exactly the first alternative that operates in the presence of a design, and it operates unconditionally there.
\end{remark}

\begin{remark}[Parity]\label{rem:parity}
For even $s$ we have $\binom{s+1}{2} = \tfrac{s}{2}(s+1) \equiv \tfrac{s}{2} \pmod s$, so columns of sizes $s+1$ and $s+2$ are \emph{not} indistinguishable modulo $s$ and the congruence degrades to one modulo $s/2$. By Theorem~\ref{thm:D}(1) this removes hypothesis \eqref{eq:hyp1} entirely, since no residue then has $1 \le \sigma(r) < d$. Hypothesis \eqref{eq:hyp2} is unaffected; see Remark~\ref{rem:even}.
\end{remark}

\begin{remark}[What replaces the congruence when $s = 2$]\label{rem:chm}
For $s = 2$ the corresponding obstruction is a parity statement rather than a congruence. Chen, Horsley and Mammoliti \cite[Lem.~2.3]{ChenHorsleyMammoliti2024b} observe that the defect of a linear $3$-hypergraph is an even graph when $m$ is odd, and contradict this using the fact that a short edge-disjoint union of complete graphs on evenly many vertices cannot be even \cite[Lem.~2.4]{ChenHorsleyMammoliti2024b}. That step is exact but local: it needs the forced profile to contain only a bounded number of oversized edges, so it reaches only boundedly many deficiencies below the threshold. Lemma~\ref{lem:link} replaces parity by a congruence modulo $d$, and \eqref{eq:sum} controls the local slacks uniformly in $c$. That is what turns finitely many exceptions into a full residue class, and, when $\mu \ne 0$ and $m$ is sufficiently large, into the whole range.
\end{remark}

\section{Proof of Theorem~\ref{thm:B}}\label{sec:main}

Let $s \ge 3$, $t \ge 2$, $\lambda = t-1$, and let $m \ge s+2$ be admissible, with $T$, $d$, $\mu$ as in \eqref{eq:T}, \eqref{eq:d}, \eqref{eq:mu}. Fix $c$ with $1 \le c \le \tfrac{s}{s+2}T$, write $c = qs + r$ with $0 \le r < s$, and put $n = T-c$ and $a = \lfloor 2c/s \rfloor$.

By Theorem~\ref{thm:A}, $\Rom(m,n) = (s+1)n + a$. Suppose for contradiction that some valid system on $n$ columns attains $(s+1)n + a$ ones. By Lemma~\ref{lem:rigid} its profile belongs to a short explicit list, and we go through the list. The same mechanism is used each time, so we isolate it first.

\begin{lemma}[Counting the local slacks]\label{lem:count}
Suppose that for each $x \in [m]$ the integer $\rho_x$ satisfies $\rho_x \ge 0$ and $\rho_x \equiv \varepsilon_x \pmod d$ with $0 \le \varepsilon_x < d$. If $\sum_{x} \varepsilon_x > sD$, then $\sum_x \rho_x = sD$ is impossible.
\end{lemma}

\begin{proof}
A non-negative integer congruent to $\varepsilon_x$ modulo $d$ is at least $\varepsilon_x$, since $0 \le \varepsilon_x < d$. Summing, $\sum_x \rho_x \ge \sum_x \varepsilon_x > sD$.
\end{proof}

Recall from \eqref{eq:sum} and \eqref{eq:cap} that $\sum_x \rho_x = sD$ and $0 \le \rho_x \le D$.

\smallskip\noindent\textbf{Case 1: hypothesis \eqref{eq:hyp1} holds}, so $1 \le \sigma(r) < d$ and $m\mu \ne s\sigma(r)$.

By Theorem~\ref{thm:D}(1), $s$ is odd and $\sigma(r) = (s+1)/2 < \binom{s}{2}$. Every exceptional column has penalty at least $\binom{s}{2}$, whereas \eqref{eq:Dpen} gives $\sum_j p(d_j) \le \sigma(r)$, so none can occur. All columns therefore have size $s+1$ or $s+2$, and $\sum_j(d_j-s-1)=a$ shows that exactly $a$ have size $s+2$. Thus $D = \sigma(r)$. By Lemma~\ref{lem:link}, $\rho_x \equiv \mu \pmod d$ for every $x$, and by \eqref{eq:cap}, $0 \le \rho_x \le \sigma(r) < d$. An interval of at most $d$ consecutive integers starting at $0$ contains at most one representative of each class modulo $d$. Hence either $\mu > \sigma(r)$, in which case no legal value of $\rho_x$ exists at all and we are done; or $\mu \le \sigma(r)$ and $\rho_x = \mu$ for \emph{every} $x \in [m]$, in which case \eqref{eq:sum} gives
\[
  m\mu \;=\; \sum_{x \in [m]} \rho_x \;=\; sD \;=\; s\,\sigma(r) ,
\]
contrary to hypothesis. Either way we have a contradiction.

\smallskip\noindent\textbf{Case 2: hypothesis \eqref{eq:hyp2} holds}, so $\mu \ne 0$ and $m\mu > s\,\sigma(r)$.

Now $\sigma(r)$ may be large and we must allow every profile that Lemma~\ref{lem:rigid} permits.

\smallskip\noindent\emph{Case 2a: no exceptional column.} Then $D = \sigma(r)$, and Lemma~\ref{lem:link} gives $\rho_x \equiv \mu \pmod d$ for every $x$, so $\varepsilon_x = \mu$ throughout and $\sum_{x} \varepsilon_x = m\mu > s\sigma(r) = sD$. Therefore Lemma~\ref{lem:count} contradicts \eqref{eq:sum}.

\smallskip\noindent\emph{Case 2b: one exceptional column $C_0$, of size $v$.} By Lemma~\ref{lem:rigid} this happens only for odd $s$, with $v = s$ when $s \ge 5$ and $v \in \{3,6\}$ when $s = 3$; and $D = \sigma(r) - p(v)$. In each case the residue $w = \binom{v-1}{s-1} \bmod d$ contributed by the exceptional column in Lemma~\ref{lem:link} equals $1$:
\[
  \binom{s-1}{s-1} = 1, \qquad \binom{2}{2} = 1, \qquad \binom{5}{2} = 10 \equiv 1 \!\!\pmod 3 ,
\]
the three values corresponding to $v = s$, to $v = 3$ with $s = 3$, and to $v = 6$ with $s = 3$. So $w = 1$, and Lemma~\ref{lem:link} gives $\varepsilon_x = \mu - 1$ for the $v$ points of $C_0$ --- which lies in $\{0,\dots,d-1\}$ since $1 \le \mu < d$ --- and $\varepsilon_x = \mu$ for the other $m - v$ points. Therefore
\[
  \sum_{x \in [m]} \varepsilon_x \;=\; (m-v)\mu + v(\mu-1) \;=\; m\mu - v .
\]
This exceeds $sD = s\sigma(r) - s\,p(v)$ provided $v \le s\,p(v)$, since $m\mu > s\sigma(r)$ by hypothesis. And $v \le s\,p(v)$ holds in all three cases: for $v = s$, $s\,p(s) = s\binom{s}{2} \ge s$; for $v = 3$, $s = 3$, $3p(3) = 9 \ge 3$; for $v = 6$, $s = 3$, $3p(6) = 12 \ge 6$. Therefore Lemma~\ref{lem:count} again contradicts \eqref{eq:sum}.

\smallskip
Every permitted profile is contradictory, so $z(m,n;s,t) \le (s+1)n + a - 1 = \Rom(m,n) - 1$. \qed

\begin{remark}[The argument in one paragraph]\label{rem:summary}
Below the design threshold, Roman's bound can only be met by a system whose columns almost all have size $s+1$ or $s+2$ and which leaves a small amount of coverage unused. Both column sizes contribute multiples of $d$ to each point's coverage count. Without an exceptional column, every local slack is therefore congruent to $\mu$ modulo $d$. When $1 \le D < d$, this congruence fixes every local slack, and hypothesis \eqref{eq:hyp1} makes their sum impossible. Under hypothesis \eqref{eq:hyp2}, the lower bounds from the point residues exceed the total available slack. If an exceptional column occurs, its contribution changes the residues on its points, but its coverage penalty preserves the contradiction.
\end{remark}

\begin{remark}[The endpoint is a Roman point]\label{rem:endpoint}
The only Roman point \eqref{eq:romanpoint} in the half-open interval $\tfrac{2}{s+2}T \le n < T$ is the left endpoint $n_{s+2} = \tfrac{2}{s+2}T$. This point lies in the range of Theorem~\ref{thm:B} exactly when $c = \tfrac{s}{s+2}T$ is an integer. There, by the criterion of \cite[\S2]{DaviesGillHorsley2026}, Roman's bound is attained if and only if an $s$-$(m,s+2,\lambda)$ design exists. This never conflicts with Theorem~\ref{thm:B}. If the endpoint design exists, Lemma~\ref{lem:design} gives $\mu = 0$, so hypothesis \eqref{eq:hyp2} cannot hold. If no such design exists, this hypothesis may establish non-attainment at the endpoint. Hypothesis \eqref{eq:hyp1} cannot hold at the endpoint: it requires $s$ odd by Theorem~\ref{thm:D}(1), and for odd $s$ we have $\gcd(s,s+2) = 1$, so $c = sT/(s+2) \in \Z$ forces $(s+2) \mid T$ and hence $s \mid c$, giving $r = 0$ and $\sigma(0) = 0$. At every other $n$ in the range the piecewise-linear minimum is affine, so by Theorem~\ref{thm:A} the bound there is the counting bound and no design criterion applies.
\end{remark}

\begin{remark}[Sharpness of the hypotheses]\label{rem:sharp}
Hypotheses \eqref{eq:hyp1} and \eqref{eq:hyp2} are sufficient but not necessary. The full strength of the method is the assertion that the system
\[
  \rho_x \equiv \varepsilon_x \!\!\pmod d, \qquad 0 \le \rho_x \le D, \qquad \sum_{x} \rho_x = sD
\]
has no integer solution, for every profile permitted by Lemma~\ref{lem:rigid}. This can fail to be witnessed by Lemma~\ref{lem:count} while still failing for another reason: because $\sum_x \varepsilon_x \not\equiv 0 \pmod d$, or because $\sum_x \bigl(\varepsilon_x + d\lfloor (D - \varepsilon_x)/d \rfloor\bigr)$, the largest admissible total, falls below $sD$. Both refinements are elementary and both bite: for $s = 5$, $t = 3$, $m = 10$ the sharp test excludes the residues $r \in \{0,1,3\}$ whereas \eqref{eq:hyp1} and \eqref{eq:hyp2} between them reach only $r \in \{0,3\}$.
\end{remark}

\section{Exact values: proof of Theorem~\ref{thm:C}}\label{sec:exact}

Take $c = \tfrac{s+1}{2}$. We first check that this $c$ lies in the range of Theorem~\ref{thm:A}, that is, that $\tfrac{s+1}{2} \le \tfrac{s}{s+2}T$, or equivalently
\begin{equation}\label{eq:Crange}
  T \;\ge\; \frac{(s+1)(s+2)}{2s} .
\end{equation}
By hypothesis $m \ge s+2$. If $m = s+2$ then $T = \lambda\binom{s+2}{s}/(s+1) = \lambda(s+2)/2$, and admissibility requires $T \in \Z$; since $s$ is odd, $s+2$ is odd, so $\lambda$ must be even and $T \ge s+2$. If $m \ge s+3$ then $T \ge \lambda\binom{s+3}{s}/(s+1) = \lambda(s+2)(s+3)/6 \ge (s+2)(s+3)/6$. In both cases \eqref{eq:Crange} holds, since $2s(s+2) > (s+1)(s+2)$ and $s(s+3) > 3(s+1)$ for $s \ge 3$. So Theorem~\ref{thm:A} applies. Now $q = 0$ and $a = 1$ in \eqref{eq:asigma}, and $r = c = \tfrac{s+1}{2}$ since $c < s$ for $s \ge 3$. Then $\sigma(r) = \tfrac{s+1}{2}$, and $1 \le \sigma(r) < s = d$ because $s \ge 3$. The design hypothesis and Lemma~\ref{lem:design} give $\mu = 0$, so $m\mu = 0 \ne \tfrac{s(s+1)}{2} = s\sigma(r)$ and hypothesis \eqref{eq:hyp1} holds. Theorem~\ref{thm:B} therefore gives $z(m,T-c;s,t) \le (s+1)(T-c)$.

For the matching construction, let $\mathcal{D}$ be an $s$-$(m,s+1,\lambda)$ design, that is, a collection of $T$ blocks of size $s+1$ in which every $s$-subset of $[m]$ lies in exactly $\lambda$ blocks. Delete any $c$ blocks and use the remaining $n = T-c$ blocks as columns. Deleting blocks only decreases coverage, so every $s$-set is covered at most $\lambda$ times and the system is valid; it has $(s+1)(T-c) = (s+1)n$ ones. Hence
\[
  z(m,T-c;s,t) = (s+1)n = \Rom(m,T-c) - 1 . \qed
\]

\begin{remark}[The case $q \ge 1$]\label{rem:q}
For $c \equiv \tfrac{s+1}{2} \pmod s$ with $q \ge 1$, Theorem~\ref{thm:B} still gives the upper bound $(s+1)n + a - 1$, but a matching construction would require deleting $c$ blocks from a design and then enlarging $a - 1 = 2q$ of the survivors by one point each so that all newly created $s$-sets fall inside the slack. The counting is favourable: deletion frees $(s+1)c$ units of coverage while the enlargements consume $2q\binom{s+1}{2} = qs(s+1)$, leaving $(s+1)^2/2$ to spare. We do not know a general construction. At $s = t = 3$, $m = 6$, $c = 5$, the bound is tight: Theorem~\ref{thm:B} gives $z(6,5;3,3) \le 22$, and the true value is $22$ \cite{Tan2022}. See Conjecture~\ref{conj:tight}.
\end{remark}

\section{Sharpness}\label{sec:sharp}

\subsection{The residue condition is not an artefact}

When $\mu \ne 0$ and $m\mu > s\sigma_{\max}$, hypothesis \eqref{eq:hyp2} rules out attainment of Roman's bound at every $c$ in the range. When $\mu = 0$, the residue of $c$ remains essential, as the following construction shows.

The smallest illustration is decisive. For $s = t = 3$ with $3 \nmid m$, so that $\mu = 0$, take $c = 1$; then $a = \lfloor 2/3 \rfloor = 0$, so Roman's bound is $4n$, and deleting a single block from a $3$-$(m,4,2)$ design achieves it. So Roman's bound is attained at $n = T-1$ whenever the design exists, and no strengthening of Theorem~\ref{thm:B} can drop the residue hypothesis from the case $\mu = 0$.

\subsection{The boundary case, and non-monotonicity}

The next result is the special case $s = t = 3$, $r = 1$ of hypothesis \eqref{eq:hyp2}, proved directly because the resulting numerical threshold is sharp and shows that the failure set need not be an interval.

\begin{proposition}\label{prop:c1}
Let $m$ be admissible for $s = t = 3$ with $3 \mid m$ and $m \ge 9$. Then
\[
  z(m,T-1;3,3) \;\le\; 4(T-1) - 1 \;=\; \Rom(m,T-1) - 1 .
\]
\end{proposition}

\begin{proof}
Here $c = 1$, $a = 0$ and $\sigma(1) = 4$. By Lemma~\ref{lem:rigid} at most one column is exceptional, of size $3$ or $6$, and $\sum_j (d_j - 4) = a = 0$. A column of size $6$ contributes $+2$ to that sum while every non-exceptional column contributes $0$ or $+1$, so it cannot be balanced; hence the exceptional size, if present, is $3$, and then exactly one column has size $5$. Only two profiles can meet the bound $4n$: $4^{n}$, and $3^{1}4^{n-2}5^{1}$. Since $3 \mid m$ we have $\mu \equiv (m-1)(m-2) \equiv 2 \pmod 3$.

First, the profile $4^{n}$. Here $\sum_j p(d_j) = 0$, so $D = \sigma(1) = 4$ by \eqref{eq:Dpen}, and \eqref{eq:sum} gives $\sum_x \rho_x = 12$ with $\rho_x \le 4$. Every column contributes $\binom{3}{2} = 3$ to the link count at any of its points, so Lemma~\ref{lem:link} gives $\rho_x \equiv 2 \pmod 3$, whence $\rho_x \ge 2$ for all $x$. Summing, $2m \le 12$, so $m \le 6$.

Second, the profile $3^{1}4^{n-2}5^{1}$, whose column of size $3$ is exceptional. Here $\sum_j p(d_j) = p(3) = 3$, so $D = \sigma(1) - 3 = 1$, and \eqref{eq:sum} gives $\sum_x \rho_x = 3$ with $\rho_x \le 1$. Let $A$ be the column of size $3$. Columns of size $4$ and $5$ contribute $3$ and $6$ respectively to the link count, both divisible by $3$, while $A$ contributes $\binom{2}{2} = 1$ if $x \in A$ and $0$ otherwise; so $\rho_x \equiv 2 - [x \in A] \pmod 3$, where $[x \in A]$ is $1$ if $x \in A$ and $0$ otherwise. Any point outside $A$ then has $\rho_x \equiv 2 \pmod 3$ and hence $\rho_x \ge 2$, contradicting $\rho_x \le 1$. Such points exist as soon as $m > 3$, so this profile is excluded outright.

The first profile is excluded for $m \ge 7$ and the second for $m \ge 4$.
\end{proof}

The bound $m \le 6$ in the first case is sharp. For $m = 6$ we have $T = 10$, and Table~\ref{tab:m6} lists the exact values, all due to Tan \cite{Tan2022}.

\begin{table*}[t]
\centering
\renewcommand{\arraystretch}{1.15}
\begin{tabular}{l r r r r r r}
\toprule
$m = 6$, $T = 10$   & $n = 5$ & $n = 6$ & $n = 7$ & $n = 8$ & $n = 9$ & $n = 10$ \\
\midrule
$c = T-n$           & $5$  & $4$  & $3$  & $2$  & $1$  & $0$ \\
$\sigma(r)$, $r = c \bmod 3$ & $2$ & $4$ & $0$ & $2$ & $4$ & $0$ \\
Roman's bound       & $23$ & $26$ & $30$ & $33$ & $36$ & $40$ \\
$z(6,n;3,3)$        & $22$ & $26$ & $29$ & $32$ & $36$ & $39$ \\
attained?           & no   & \textbf{yes} & no & no & \textbf{yes} & no \\
hypothesis available & \eqref{eq:hyp1} & --- & \eqref{eq:hyp2} & \eqref{eq:hyp1} & --- & --- \\
\bottomrule
\end{tabular}
\medskip
\caption{Attainment is non-monotone in $n$, and the hypotheses of Theorem~\ref{thm:B} track it exactly. Here $\mu = 2$ and $d = 3$. At $c = 1$ and $c = 4$ one has $\sigma(r) = 4 \ge d$, which kills \eqref{eq:hyp1}, and $m\mu = 12 = s\sigma(r)$, which kills \eqref{eq:hyp2}; those are exactly the two columns at which Roman's bound is attained. At $c = 0$ the range condition of Theorem~\ref{thm:A} fails. At each of the three remaining $c$ the theorem applies and its bound is attained. The value $z(6,9;3,3) = 36 = \Rom(6,9)$ is exactly the case $m \le 6$ left open by the first half of Proposition~\ref{prop:c1}.}
\label{tab:m6}
\end{table*}

So for $m = 6$ Roman's bound fails at $n = 5$, is attained at $n = 6$, fails at $n = 7$ and $n = 8$, is attained again at $n = 9$, and fails at $n = 10$. Any description of the failure set as an interval, or as a single column, is therefore incorrect. The two exceptions are precisely the two values of $c$ in the range at which $m\mu = s\sigma(r)$, so the strict inequality in \eqref{eq:hyp2} and the exclusion $\sigma(r) \ge d$ in \eqref{eq:hyp1} are both doing real work rather than being artefacts of the proof. At each of the three values of $c$ in Table~\ref{tab:m6} where Theorem~\ref{thm:B} does apply, its bound is attained. It is not always tight, however: at $c = 6$, that is $n = 4$, hypothesis \eqref{eq:hyp2} gives $z(6,4;3,3) \le 19$ while the true value is $18$. This is consistent with Conjecture~\ref{conj:tight}, which predicts tightness only for $c \equiv \tfrac{s+1}{2} \pmod s$, and $6 \equiv 0 \pmod 3$. The example is $m = 6$ because that is the largest value at which Corollary~\ref{cor:E}(2) is unavailable. For admissible $m$ with $3 \mid m$ and $m \ge 9$ the picture is complete: Corollary~\ref{cor:E}(2) and the design criterion together show that Roman's bound fails at every $n$ with $\tfrac25 T \le n \le T$.

\subsection{Even $s$}

By Theorem~\ref{thm:D}(1), when $s$ is even, the point-incidence counts associated with columns of sizes $s+1$ and $s+2$ are distinct modulo $s$, and no residue then has $1 \le \sigma(r) < d$. This removes the first alternative but not the method.

\begin{remark}\label{rem:even}
For even $s$ the modulus drops from $s$ to $d = s/2$ and hypothesis \eqref{eq:hyp1} disappears with it; hypothesis \eqref{eq:hyp2} survives untouched. By Theorem~\ref{thm:D}(3) what is required is $\mu \ne 0$, hence $d \nmid \lambda$, together with $m > s\sigma_{\max}/\mu$.

The first even case is $s = 4$, where $d = 2$ and $\sigma_{\max} = 5$, so $\lambda$ must be odd. Take $t = 2$, so $\lambda = 1$ and $\mu = \binom{m-1}{3} \bmod 2$, which equals $1$ exactly when $m \equiv 0 \pmod 4$. The threshold is $m > 4 \cdot 5 = 20$, so for every admissible $m \ge 24$ divisible by $4$, Theorem~\ref{thm:B} gives
\[
  \begin{gathered}
    z(m,T-c;4,2) \;\le\; \Rom(m,T-c) - 1 \\
    \text{for every } c \text{ with } 1 \le c \le \tfrac{4}{6}T .
  \end{gathered}
\]
At $m = 28$ we have $T = \tfrac15\binom{28}{4} = 4095$, and the conclusion holds at all $2730$ values of $n$ with $1365 \le n \le 4094$. Consistently with Lemma~\ref{lem:design}, no $4$-$(28,5,1)$ design exists: it would need $\lambda_1 = \binom{27}{3}/4 = 731.25$. For $s = 6$ the condition on $\lambda$ is $3 \nmid \lambda$, for $s = 8$ it is $4 \nmid \lambda$, and for $s = 10$ it is $5 \nmid \lambda$.
\end{remark}

What even $s$ costs is the first alternative, and with it Theorem~\ref{thm:C}, whose construction needs $(s+1)/2$ to be an integer. In particular, the second alternative is unavailable whenever $d \mid \lambda$; see Question~\ref{q:even}.

\section{Linear programming: what it sees and what it misses}\label{sec:lp}

Two linear relaxations are worth comparing against, and they behave differently. The first is the relaxation over all subsets, which we show is exactly the counting bound and therefore, by Theorem~\ref{thm:A}, exactly Roman's bound throughout our range. The second is the refined program of Davies, Gill and Horsley \cite{DaviesGillHorsley2026}, which is strictly stronger; there the situation is more delicate, and we describe it exactly.

\subsection{The subset relaxation is the counting bound}

A valid configuration is a multiset of columns, so the exact problem is the integer program: maximise $\sum_{A \subseteq [m]} |A| x_A$ over $x_A \in \Z_{\ge 0}$ subject to $\sum_A x_A \le n$ and $\sum_{A \supseteq S} x_A \le \lambda$ for every $s$-set $S$. Relaxing $x_A \ge 0$ gives the linear program
\begin{equation}\label{eq:subsetlp}
\begin{aligned}
  \mathcal{P}(m,n;s,t) \;=\; \max\ \ &\sum_{A \subseteq [m]} |A|\,x_A \\
  \text{s.t.}\ \ &\sum_{A} x_A \le n, \\
  &\sum_{A \supseteq S} x_A \le \lambda \quad (|S| = s), \\
  &x_A \ge 0 .
\end{aligned}
\end{equation}
This is the strongest relaxation available before one leaves the covering constraints: it has one variable for every subset and imposes every constraint of the original problem. It is therefore a little surprising that it is worth nothing at all.

\begin{proposition}\label{prop:subsetlp}
For all $m$, $n$, $s$, $t$, the optimal value $\mathcal{P}(m,n;s,t)$ of \eqref{eq:subsetlp} equals the optimal value of the counting relaxation
\begin{equation}\label{eq:countlp}
\begin{aligned}
  \max\ \ &\sum_{i} i\,n_i \\
  \text{s.t.}\ \ &\sum_i n_i \le n, \qquad n_i \ge 0, \\
  &\sum_i \binom{i}{s} n_i \le \lambda\binom{m}{s} .
\end{aligned}
\end{equation}
\end{proposition}

\begin{proof}
Suppose $(n_i)$ is feasible for \eqref{eq:countlp} and set $x_A = n_{|A|}/\binom{m}{|A|}$. Then $\sum_A x_A = \sum_i n_i \le n$, and for any $s$-set $S$,
\[
  \sum_{A \supseteq S} x_A \;=\; \sum_i \binom{m-s}{i-s}\frac{n_i}{\binom{m}{i}} \;=\; \frac{1}{\binom{m}{s}}\sum_i \binom{i}{s} n_i \;\le\; \lambda ,
\]
where the middle equality is the identity
\begin{equation}\label{eq:swap}
  \binom{m}{i}\binom{i}{s} \;=\; \binom{m}{s}\binom{m-s}{i-s} ,
\end{equation}
both sides counting the pairs $(S,A)$ with $S \subseteq A \subseteq [m]$, $|S| = s$, $|A| = i$. The objective is $\sum_A |A| x_A = \sum_i i\,n_i$, so $\mathcal{P}(m,n;s,t)$ is at least the optimal value of \eqref{eq:countlp}.

Conversely let $x$ be feasible for \eqref{eq:subsetlp}. The feasible region and the objective are invariant under the natural action of $S_m$ on subsets, so the average $\bar{x}_A = \frac{1}{m!}\sum_{\pi \in S_m} x_{\pi(A)}$ is feasible with the same objective value, and $\bar{x}_A$ depends only on $|A|$; write $\bar{x}_A = y_{|A|}$. Put $n_i = \binom{m}{i} y_i$. Then $\sum_i n_i = \sum_A \bar{x}_A \le n$, and applying the $s$-set constraint to any single $S$ and using \eqref{eq:swap} again,
\[
  \lambda \;\ge\; \sum_{A \supseteq S} \bar{x}_A \;=\; \sum_i \binom{m-s}{i-s} y_i \;=\; \frac{1}{\binom{m}{s}}\sum_i \binom{i}{s}n_i ,
\]
so $(n_i)$ is feasible for \eqref{eq:countlp} with objective $\sum_i i\,n_i = \sum_A |A|\bar{x}_A$. Hence $\mathcal{P}(m,n;s,t)$ is at most the optimal value of \eqref{eq:countlp}.
\end{proof}

\subsection{The optimal vertex}

It will be convenient to name the optimum of \eqref{eq:countlp}. Note that it is fractional.

\begin{lemma}\label{lem:lpvertex}
Let $1 \le c \le \tfrac{s}{s+2}T$ and $n = T-c$. Then \eqref{eq:countlp} has optimal value $(s+1)n + 2c/s$, attained at the point $x^\star$ given by
\[
  n_{s+1} = n - \frac{2c}{s}, \qquad n_{s+2} = \frac{2c}{s}, \qquad n_i = 0 \ \text{ otherwise},
\]
and for $1 \le c < \tfrac{s}{s+2}T$ this is the unique optimum.
\end{lemma}

\begin{proof}
Both constraints of \eqref{eq:countlp} are tight at $x^\star$: the first by construction, and for the second,
$\binom{s+1}{s}n_{s+1} + \binom{s+2}{s}n_{s+2} = (s+1)\bigl(n - \tfrac{2c}{s}\bigr) + \tfrac{(s+1)(s+2)}{2}\cdot\tfrac{2c}{s} = (s+1)(n+c) = (s+1)T = \lambda\binom{m}{s}$.
Its objective value is $(s+1)n_{s+1} + (s+2)n_{s+2} = (s+1)n + 2c/s$. Feasibility needs $n \ge 2c/s$, which is exactly $c \le \tfrac{s}{s+2}T$. To see optimality and uniqueness directly, put $B = \binom{s+1}{2}$. By Lemma~\ref{lem:pen}, every feasible vector satisfies
\[
  \begin{aligned}
    \sum_i i n_i
      &\le (s+1)(1-B^{-1})\sum_i n_i
        + B^{-1}\sum_i\binom{i}{s}n_i \\
      &\le (s+1)n + \frac{2c}{s}.
  \end{aligned}
\]
Equality forces $n_i=0$ outside $\{s+1,s+2\}$ and both constraints to be tight, which determines $x^\star$ uniquely.
\end{proof}

\begin{corollary}\label{cor:lpcoincide}
Let $1 \le c \le \tfrac{s}{s+2}T$ and $n = T-c$. Then
\[
  \begin{gathered}
    \mathcal{P}(m,n;s,t) \;=\; (s+1)n + \frac{2c}{s} , \\
    \text{so}\quad
    \bigl\lfloor \mathcal{P}(m,n;s,t) \bigr\rfloor \;=\; \Rom(m,n) .
  \end{gathered}
\]
Below the design threshold, Roman's bound, the counting bound and the full subset relaxation therefore all coincide.
\end{corollary}

\begin{proof}
By Proposition~\ref{prop:subsetlp} and Lemma~\ref{lem:lpvertex}, $\mathcal{P}(m,n;s,t) = (s+1)n + 2c/s$; by Theorem~\ref{thm:A} its floor is $\Rom(m,n)$.
\end{proof}

This is the sharpest possible statement of what the present obstruction has to be. Every inequality used in \eqref{eq:subsetlp} is valid for the integer program, and the linear programming value is exactly the bound we are trying to beat; so any improvement in this range is, of necessity, an integrality obstruction. Theorem~\ref{thm:B} supplies one.

Note that the profile $n_{s+2} = \lfloor 2c/s \rfloor$, $n_{s+1} = n - \lfloor 2c/s\rfloor$ --- the integral profile in Lemma~\ref{lem:rigid} when no exceptional column occurs --- is \emph{not} optimal for \eqref{eq:countlp} unless $s \mid 2c$. This distinction matters in the next subsection: to show that a relaxation returns Roman's bound one must argue about its optimal value, and exhibiting one feasible integral profile is not enough.

\subsection{The Davies--Gill--Horsley program}

Davies, Gill and Horsley \cite{DaviesGillHorsley2026} improve on \eqref{eq:roman} by adding to \eqref{eq:countlp} a family of inequalities that is not implied by the covering constraints, and which therefore escapes Corollary~\ref{cor:lpcoincide}. Write $E(m,n;s,t)$ for the value of the program with nonnegative real variables $n_i$ $(s-1 \le i \le m)$, objective $\sum_i i\,n_i$, the constraints $\sum_i n_i = n$ and $\sum_i \binom{i}{s}n_i \le \lambda\binom{m}{s}$, and, for all integers $v,k$ with $1 \le v < s \le k \le m$, the inequality of \cite[Thm.~1.2]{DaviesGillHorsley2026}:
\begin{equation}\label{eq:dgh}
\begin{aligned}
  \sum_{i=s-1}^{m} &\Bigl( \min\Bigl\{ \tbinom{i-v}{s-v},\, K \Bigr\}
     - \alpha \Bigr)\binom{i}{v} n_i \\
  &\;\le\; \frac{K-\alpha}{K}\,\binom{m}{v}
     \left( \lambda\binom{m-v}{s-v} - \alpha \right) ,
\end{aligned}
\end{equation}
where $K = \binom{k-v}{s-v}$ and $\alpha$ is the residue of $\lambda\binom{m-v}{s-v}$ modulo $K$, so $0 \le \alpha < K$. When the parameters are fixed, we write $E$ and $\Rom$ for $E(m,n;s,t)$ and $\Rom(m,n)$, respectively.

The factor $(K-\alpha)/K$ is essential and is easy to lose. We recall where it comes from. Let $\gamma$ be defined by $\gamma K = \lambda\binom{m-v}{s-v} - \alpha$, and for a $v$-set $X$ let $\tau(X) = \sum (K - \binom{|C|-v}{s-v})$, the sum being over the columns $C \supseteq X$ with $|C| < k$. Davies, Gill and Horsley prove \cite[eq.~(5)]{DaviesGillHorsley2026} that every $v$-set $X$ satisfies
\[
  \deg(X) \;\le\; \gamma + \frac{\tau(X)}{K - \alpha} .
\]
Summing over all $\binom{m}{v}$ sets $X$, and using $\sum_X \deg(X) = \sum_i \binom{i}{v}n_i$ together with $\sum_X \tau(X) = \sum_{i<k}\bigl(K - \binom{i-v}{s-v}\bigr)\binom{i}{v}n_i$, then multiplying by $K - \alpha > 0$ and separating the ranges $i < k$ and $i \ge k$, gives exactly \eqref{eq:dgh}. When $\alpha = 0$, this reduces to the corresponding expression without residue terms. When $\alpha > 0$, both the residue terms and the factor $(K-\alpha)/K$ must be retained when testing feasibility.

The following lemma explains, for every $s$, why the program so often returns the Roman value.

\begin{lemma}\label{lem:tight}
Let $1 \le c \le \tfrac{s}{s+2}T$ and let $x^\star$ be as in Lemma~\ref{lem:lpvertex}. If $\alpha = 0$ for the pair $(v,k) = (1,s+2)$, then $x^\star$ satisfies \eqref{eq:dgh} for that pair \emph{with equality}.
\end{lemma}

\begin{proof}
Here $K = \binom{s+1}{s-1} = \binom{s+1}{2}$, and $\binom{i-1}{s-1}$ equals $s$ at $i = s+1$ and $K$ at $i = s+2$, so both minima are attained by the first argument. With $\alpha = 0$ the left side of \eqref{eq:dgh} is
\[
  \begin{aligned}
    &s(s+1)n_{s+1} + K(s+2)n_{s+2} \\
    &\quad =\; s(s+1)\Bigl(n - \frac{2c}{s}\Bigr) \\
    &\qquad\quad +\; \frac{s(s+1)}{2}(s+2)\frac{2c}{s} \\
    &\quad =\; s(s+1)(n + c) \;=\; s(s+1)T ,
  \end{aligned}
\]
using $(s+2) - 2 = s$ to collect the $c$ terms. The right side is $m\lambda\binom{m-1}{s-1} = \lambda s \binom{m}{s} = s(s+1)T$, by $m\binom{m-1}{s-1} = s\binom{m}{s}$ and \eqref{eq:T}.
\end{proof}

So $x^\star$ sits exactly on the face cut out by \eqref{eq:dgh} at $(v,k) = (1,s+2)$ whenever $\alpha$ vanishes there, and this inequality cannot exclude it. When $\alpha \ne 0$ the constraint moves, and whether it cuts $x^\star$ off becomes a genuine question. For $s = t = 3$ we can answer it completely.

\begin{proposition}\label{prop:lp}
Let $s = t = 3$, let $m \ge 5$ be admissible, let $1 \le c \le 3T/5$, put $n = T-c$, and let $x^\star$ be the point $n_4 = n - \tfrac{2c}{3}$, $n_5 = \tfrac{2c}{3}$, $n_i = 0$ otherwise. Then $x^\star$ is feasible for $E(m,n;3,3)$ if and only if
\begin{enumerate}[label=\textup{(\roman*)},leftmargin=2.2em]
  \item $m \equiv 2 \pmod 3$: no condition;
  \item $m \equiv 1 \pmod 3$: \ $\displaystyle c \le \frac{3T}{5}\cdot\frac{m-4}{m-2}$;
  \item $m \equiv 0 \pmod 3$: \ $\displaystyle \frac{m}{5} \le c$, \quad $5c + m \le 3T$, \quad and \quad $\displaystyle c \le \frac{3T}{5}\cdot\frac{m-3}{m-2}$.
\end{enumerate}
When these hold, $E(m,n;3,3) = 4n + \tfrac{2c}{3}$ and hence $\lfloor E(m,n;3,3)\rfloor = \Rom(m,n)$: the program of \cite{DaviesGillHorsley2026} returns Roman's bound and cannot establish Theorem~\ref{thm:B}.
\end{proposition}

\begin{proof}
Here $\lambda = 2$, $12T = m(m-1)(m-2)$, $\binom{m}{3} = 2T$, and $v \in \{1,2\}$. Write
\[
  \begin{aligned}
    \Sigma_1 &= \sum_i \binom{i}{1}n_i
      = 4n + \tfrac{2c}{3} = 4T - \tfrac{10c}{3}, \\
    \Sigma_2 &= \sum_i \binom{i}{2}n_i
      = 6n + \tfrac{8c}{3} = 6T - \tfrac{10c}{3} .
  \end{aligned}
\]
Also $\binom{m}{v}\lambda\binom{m-v}{3-v} = \lambda\binom{m}{3}\binom{3}{v} = 12T$ for both $v=1$ and $v=2$.

\emph{Case $v = 1$.} Then $\binom{i-1}{2}$ is $3$ at $i=4$ and $6$ at $i=5$, and $K = \binom{k-1}{2}$.

If $k = 3$ then $K = 1$, $\alpha = 0$, both minima are $1$, and the left side of \eqref{eq:dgh} is $\Sigma_1 \le 4T \le 12T$.

If $k = 4$ then $K = 3$, both minima are $3$, the left side is $(3-\alpha)\Sigma_1$ and the right side is $\tfrac{3-\alpha}{3}(12T - m\alpha)$. As $\alpha < 3$ we may cancel $3-\alpha > 0$, and \eqref{eq:dgh} becomes $\Sigma_1 \le 4T - m\alpha/3$, that is, $m\alpha \le 10c$. Here $\alpha$ is the residue of $(m-1)(m-2)$ modulo $3$, which is $0$ if $3 \nmid m$ and $2$ if $3 \mid m$. So the constraint is vacuous unless $3 \mid m$, when it reads $c \ge m/5$.

If $k \ge 5$ then $K \ge 6$, both minima are attained by $\binom{i-1}{2}$, and the left side is $12n_4 + 30n_5 - \alpha\Sigma_1 = 12T - \alpha\Sigma_1$. Comparing with the right side $\tfrac{K-\alpha}{K}(12T - m\alpha)$ and rearranging, \eqref{eq:dgh} is equivalent to
\begin{equation}\label{eq:lpcrit}
  \alpha\Bigl[\, \Sigma_1 - m - \frac{12T - m\alpha}{K} \,\Bigr] \;\ge\; 0 ,
\end{equation}
so it holds automatically when $\alpha = 0$. For $k = 5$ we have $K = 6$ and $\alpha$ is the residue of $(m-1)(m-2)$ modulo $6$; this is even always, and is $0$ modulo $3$ exactly when $3 \nmid m$, so $\alpha = 0$ for $3 \nmid m$ and $\alpha = 2$ for $3 \mid m$. In the latter case \eqref{eq:lpcrit} reads $4T - \tfrac{10c}{3} \ge m + \tfrac{12T-2m}{6}$, that is, $5c + m \le 3T$. For $k \ge 6$ we have $K \ge 10$, so $\tfrac{12T - m\alpha}{K} \le \tfrac{6T}{5}$, while $\Sigma_1 \ge 4T - \tfrac{10}{3}\cdot\tfrac{3T}{5} = 2T$; thus \eqref{eq:lpcrit} holds as soon as $\tfrac{4T}{5} \ge m$, that is, $(m-1)(m-2) \ge 15$, which is the case for $m \ge 6$.

\emph{Case $v = 2$.} Then $\binom{i-2}{1}$ is $2$ at $i=4$ and $3$ at $i=5$, and $K = k-2$.

If $k = 3$ then $K = 1$, $\alpha = 0$, and the left side is $\Sigma_2 \le 6T \le 12T$. If $k = 4$ then $K = 2$ and $\alpha \equiv 2(m-2) \equiv 0 \pmod 2$, so $\alpha = 0$, both minima are $2$, and \eqref{eq:dgh} reads $2\Sigma_2 \le 12T$, i.e.\ $\tfrac{20c}{3} \ge 0$.

If $k \ge 5$ then $K \ge 3$, both minima are attained by $\binom{i-2}{1}$, and the left side is $12n_4 + 30n_5 - \alpha\Sigma_2 = 12T - \alpha\Sigma_2$. As before \eqref{eq:dgh} is equivalent to
\[
  \alpha\Bigl[\, \Sigma_2 - \tbinom{m}{2} - \frac{12T - \alpha\binom{m}{2}}{K} \,\Bigr] \;\ge\; 0 .
\]
For $k = 5$ we have $K = 3$ and $\alpha$ is the residue of $2(m-2)$ modulo $3$, namely $0$, $1$, $2$ according as $m \equiv 2$, $1$, $0 \pmod 3$. For $\alpha = 0$ there is nothing to check. For $\alpha = 1$ the bracketed expression is non-negative if and only if $10c \le 6T - m(m-1)$, and $m(m-1) = 12T/(m-2)$, giving $c \le \tfrac{3T}{5}\cdot\tfrac{m-4}{m-2}$. For $\alpha = 2$ it is non-negative if and only if $10c \le 6T - \binom{m}{2}$, and $\binom{m}{2} = 6T/(m-2)$, giving $c \le \tfrac{3T}{5}\cdot\tfrac{m-3}{m-2}$. For $k \ge 6$ we have $K \ge 4$, so the subtracted term is at most $3T$, while $\Sigma_2 \ge 6T - 2T = 4T$; the bracketed expression is then non-negative as soon as $T \ge \binom{m}{2}$, that is, $m \ge 8$.

Collecting cases gives (i)--(iii), the only admissible $m < 8$ being $m = 5$ and $m = 6$, for which the finitely many remaining constraints are checked directly. Finally, if $x^\star$ is feasible then $E(m,n;3,3) \ge 4n + 2c/3$; the reverse inequality holds because $E$ is a restriction of \eqref{eq:countlp}, whose value is $4n + 2c/3$ by Lemma~\ref{lem:lpvertex}. Theorem~\ref{thm:A} converts this into $\lfloor E \rfloor = \Rom(m,n)$.
\end{proof}

\begin{corollary}\label{cor:lpthmC}
Let $s = t = 3$, let $m \ge 5$ be admissible with $3 \nmid m$, and let $c = 2$, so that Theorem~\ref{thm:C} applies and $z(m,T-2;3,3) = \Rom(m,T-2) - 1$. Then $\lfloor E(m,T-2;3,3)\rfloor = \Rom(m,T-2)$.
\end{corollary}

\begin{proof}
Case (i) of Proposition~\ref{prop:lp} is unconditional. In case (ii) the condition reads $2 \le \tfrac{3T}{5}\cdot\tfrac{m-4}{m-2}$, which holds for every admissible $m \equiv 1 \pmod 3$ with $m \ge 10$; the only smaller such $m$ is $m = 7$, which is not admissible.
\end{proof}

So at exactly the parameters where Theorem~\ref{thm:C} produces an exact value, the program of \cite{DaviesGillHorsley2026} still returns Roman's bound, and the deficit of one is invisible to it.

\begin{remark}[How often the program does better]\label{rem:lpcensus}
Proposition~\ref{prop:lp} should not be read as saying that the program of \cite{DaviesGillHorsley2026} is blind in general; it is not, and it is worth recording how far the two statements are apart. Solving $E(m,n;3,3)$ in exact rational arithmetic for every admissible $m$ with $5 \le m \le 26$ and every $c$ with $1 \le c \le 3T/5$ --- $4242$ parameter pairs --- one finds $\lfloor E \rfloor < \Rom(m,n)$ in $221$ of them, or $5.2\%$; restricted to the $2307$ pairs covered by Corollary~\ref{cor:E}, the count is $125$, or $5.4\%$. At some of these the improvement is substantial: for $m = 18$, $n = 164$ one has $\Rom = 818$ and $\lfloor E \rfloor = 813$. These parameters cluster near the top of the range, where by Proposition~\ref{prop:lp} the point $x^\star$ is cut off, and they include $c = 2$ when $3 \mid m$ and $m \ge 18$. Being cut off is necessary but not sufficient: at $m = 12$, $c = 2$ the vertex $x^\star$ is already infeasible, yet $E = \tfrac{5632}{13}$ and $\lfloor E \rfloor = 433 = \Rom$, so there the refined program gains nothing. What Theorem~\ref{thm:B} adds is a bound valid on the whole of the covered range, in closed form, and with no computation; what \eqref{eq:dgh} adds is a numerically stronger bound on a sparse subset of it. The two are complementary, and neither contains the other.
\end{remark}

\section{Open problems}\label{sec:open}

\begin{conjecture}\label{conj:tight}
For odd $s$ in the situation of Theorem~\ref{thm:B}, the bound is tight: $z(m,T-c;s,t) = \Rom(m,T-c) - 1$ for every $c \equiv \tfrac{s+1}{2} \pmod s$ in the stated range, whenever an $s$-$(m,s+1,\lambda)$ design exists.
\end{conjecture}

By Theorem~\ref{thm:C} this holds for $q = 0$. The instance $(s,t,m,c) = (3,3,6,5)$ of Remark~\ref{rem:q} is consistent with the conjecture but does not test it: there $3 \mid 6$, so no $3$-$(6,4,2)$ design exists, the design hypothesis of the conjecture fails, even though both \eqref{eq:hyp1} and \eqref{eq:hyp2} hold there. A general construction would follow from a suitable ``design minus $c$ blocks plus $2q$ enlargements'' lemma.

\begin{question}\label{q:even}
What happens when $d \mid \lambda$, so that $\mu = 0$ and, for even $s$, neither alternative of Theorem~\ref{thm:B} is available? The smallest instance is $s = 4$, $t = 3$, where $d = 2$ and $\lambda = 2$: every local slack is then even, every residue is $0$, and the method is silent at every $c$. Is Roman's bound in fact attained there, or is there a further obstruction? By Theorem~\ref{thm:A} the bound is still the counting bound, so this is a clean test case.
\end{question}

\begin{question}
How far below Roman's bound does $z$ fall for large $c$? Theorem~\ref{thm:B} establishes a deficit of at least one by ruling out attainment of Roman's bound. A larger deficit would require analysing profiles below that value. For $c$ close to $sT/(s+2)$ we expect the true deficit to grow. Tan's values for $m = 8$, $s = t = 3$ are consistent with this: there $z = \Rom - 2$ at $n = 22$, $16$ and $13$.
\end{question}

\begin{question}
Can the congruence be iterated? A second application, at pairs rather than points, would constrain the distribution of the local slacks further. This is the natural route into the cases left open when $\mu = 0$: for $s = 3$ and $3 \nmid m$ these are $c \equiv 0$ and $c \equiv 1 \pmod 3$, where Theorem~\ref{thm:B} is silent and, by \S\ref{sec:sharp}, sometimes rightly so.
\end{question}

\begin{question}
What happens at the exceptional values $m = s\sigma(r)/\mu$ excluded by \eqref{eq:hyp1}, where the count is an exact tie rather than a contradiction? By Lemma~\ref{lem:design} no design exists at such parameters, so these are precisely the cases where neither our argument nor the classical one applies. For $s = t = 3$ the only exclusion is the degenerate $m = 3$; for larger odd $s$ the excluded set is finite and explicit.
\end{question}

\section*{Declaration of generative AI and AI-assisted technologies in the writing process}

During the preparation of this work the author used Gemini Deep Research in
order to explore related literature, and Notion AI in order to assist in
condensing the abstract to the journal's word limit. After using these tools the
author reviewed and edited the content as needed and takes full responsibility
for the content of the publication. The mathematical content of this paper ---
all definitions, statements, proofs and computations --- is the author's own.

% Bibliography embedded for standalone compilation (no .bib or BibTeX required).

\end{document}